\documentclass[12pt]{amsart}
\usepackage{amsmath,amsthm,amsfonts,amssymb,eucal}

\renewcommand{\b}{\beta}

\newcommand{\g}{\gamma}

\renewcommand{\d}{\delta}

\newcommand{\z}{\zeta}
\renewcommand{\t}{\theta}

\newcommand{\p}{\partial}

\newcommand{\R}{ \mathbb R}

\newcommand{\SN}{{\sf{N}}}

\newcommand {\bxi}{\boldsymbol\xi}

\newcommand{\CB}{\mathcal B}

\newcommand{\CD}{\mathcal D}

\newcommand{\plainW}[2]{\textup{{\textsf{W}}}^{#1, #2}}
\newcommand{\plainC}[1]{\textup{{\textsf{C}}}^{#1}}

\newcommand{\plainL}[1]{\textup{{\textsf{L}}}^{#1}}
\newcommand{\plainl}[1]{\textup{{\textsf{l}}}^{#1}}

\DeclareMathOperator{\tr}{{tr}}

\newcommand{\1}
{{\,\vrule depth3pt height9pt}{\vrule depth3pt height9pt}
{\vrule depth3pt height9pt}{\vrule depth3pt height9pt}\,}

\DeclareMathOperator{\supp}{{supp}}

\newtheorem{thm}{Theorem}[section]
\newtheorem{cor}[thm]{Corollary}
\newtheorem{lem}[thm]{Lemma}
\newtheorem{prop}[thm]{Proposition}
\newtheorem{cond}[thm]{Condition}

\theoremstyle{definition}

\newtheorem{rem}[thm]{Remark}

\numberwithin{equation}{section}

\newcommand{\bee}{\begin{equation}}
\newcommand{\ene}{\end{equation}}
\newcommand{\bees}{\begin{equation*}}
\newcommand{\enes}{\end{equation*}}
\newcommand{\bes}{\begin{split}}
\newcommand{\ens}{\end{split}}

\newcommand{\bet}{\begin{thm}}
\newcommand{\ent}{\end{thm}}
\newcommand{\bel}{\begin{lem}}
\newcommand{\enl}{\end{lem}}
\newcommand{\bec}{\begin{cor}}
\newcommand{\enc}{\end{cor}}
\newcommand{\bep}{\begin{proof}}
\newcommand{\enp}{\end{proof}}
\newcommand{\ber}{\begin{rem}}
\newcommand{\enr}{\end{rem}}

\newcommand{\Z}{\mathbb Z}

\begin{document}
\hoffset -4pc

\title
[ Trace formulas]
{{On a coefficient 
in trace formulas for Wiener-Hopf operators}}
 \author{Alexander V. Sobolev}
 \address{Department of Mathematics\\ University College London\\
Gower Street\\ London\\ WC1E 6BT UK}
\email{a.sobolev@ucl.ac.uk}

\begin{abstract}
Let $a = a(\xi), \xi\in\R,$ 
be a smooth function quickly decreasing at infinity. 
For the Wiener-Hopf operator $W(a)$ with the  
symbol $a$, and a smooth 
function $g:\mathbb C\to~\mathbb C$, H. Widom in 1982 established the following trace formula:
\[
\tr\bigl(g\bigl(W(a)\bigr) - W(g\circ a)\bigr) = \CB(a; g),
\]
where $\CB(a; g)$ is given explicitly in terms of the functions 
$a$ and $g$. The paper analyses the coefficient $\CB(a; g)$ 
for a class of non-smooth functions $g$ assuming that $a$ is real-valued. 
A representative example of one such function is $g(t) = |t|^{\g}$ with some 
$\g\in (0, 1]$. 
\end{abstract}

\maketitle

\section{Introduction}

Let $a:\R\to\mathbb C$ be a function. 
On $\plainL2(\R_+)$, $\R_+ = (0, \infty),$ 
define the Wiener-Hopf operator $W(a)$ with symbol $a$ by 
\begin{equation*}
\bigl(W(a)u\bigr)(x)
= \chi_+(x) \frac{1}{2\pi}\int e^{i(x-y)\xi} a(\xi) \chi_+(y)u(y) dy d\xi,\ 
u\in \plainL2(\R_+),
\end{equation*}  
where $\chi_+$ is the indicator of the half-line $\R_+$. 
If the limits are not specified, we always assume that 
the integration is taken over the entire line. 
We are interested in the operator 
\begin{equation}\label{WH:eq}
g\bigl(W(a)\bigr) - W(g\circ a),
\end{equation}
with a suitable function $g:\mathbb C\to\mathbb C$. 
In \cite{Widom_82}, 
see also \cite{Widom_87}, H. Widom proved that this operator is trace class if 
\begin{equation}\label{sobolev:eq}
a\in\plainL\infty(\R),\ \ 
	\iint\frac{|a(\xi_1) - a(\xi_2)|^2}{|\xi_1-\xi_2|^2} d\xi_1 d\xi_2<\infty, 
	\end{equation}
and established the following remarkable trace formula for the operator in \eqref{WH:eq}. 
 For any function $g: \mathbb C\to\mathbb C$ and any $s_1, s_2\in\mathbb C$ 
 denote
 \begin{equation}\label{U:eq}
 U(s_1, s_2; g) 
 = \int_0^1 \frac{g\bigl((1-t)s_1 + t s_2\bigr) 
 	- [(1-t)g(s_1) + t g(s_2)]}{t(1-t)} dt,
 \end{equation} 
 and introduce 
 \begin{equation}\label{cb:eq}
 \CB(a; g) = \frac{1}{8\pi^2}\iint 
 \frac{U\bigl(a(\xi_1), a(\xi_2); g\bigr)}{|\xi_1-\xi_2|^2}
 d\xi_1 d\xi_2.
 \end{equation}
 Both objects are well-defined under the conditions of the next proposition:
 
 \begin{prop}\label{Widom_82:prop}
 	[see \cite{Widom_82}, Theorem 1(a)]  
 	Suppose that 
 	 \eqref{sobolev:eq} is satisfied, 
 	and let $g$ be analytic on a 
 	neighbourhood of the closed convex hull of 
 	the function $a$. Then the operator 
 	\eqref{WH:eq} is trace class and 
 	\begin{equation}\label{Widom_82:eq}
 	\tr \bigl[
 	g\bigl(W(a)\bigr) - W(g\circ a)\bigr]
  	= \CB(a; g).
 	\end{equation}
 \end{prop}

If $a$ real-valued, then the analyticity assumptions on $g$ can be replaced by some finite smoothness, see \cite{Widom_82}, Theorem 1(b). 
In paper \cite{Peller} the assumptions on $a$ and $g$ are relaxed even further: 
the formula  \eqref{Widom_82:eq} is proved for real-valued $a$ 
under the assumptions that the integral in \eqref{sobolev:eq} is finite 
and $g$ belongs to the Besov class $B^{2}_{\infty, 1}(\R)$.

The quantity $\CB(a; g)$ is an object that one encounters very often 
in the theory of Wiener-Hopf operators. 
It appears e.g. in \cite{Peller}, \cite{Roc}, 
\cite{Widom_80}, \cite{Widom_82}, \cite{Widom_85}, \cite{Widom_87} 
as an asymptotic coefficient in various trace 
formulas for truncated Wiener-Hopf and Toeplitz operators with 
smooth symbols. 
Moreover, the function $U(s_1, s_2; g)$ is 
present in a variety of trace formulas for 
the same operators with discontinuous symbols, 
see e.g. \cite{Basor}, \cite{Widom_821}, \cite{Sob}, \cite{Sob2} and references therein.  
Although the integral \eqref{U:eq} is well-defined for rather a wide class of functions $g$, 
the coefficient \eqref{cb:eq} itself has been considered so far for smooth functions $g$ only. 
 As observed in \cite{Widom_82}, 
 if $g$ is twice differentiable, we can integrate by parts in \eqref{U:eq} to obtain that 
 \begin{equation*}
 	U(s_1, s_2; g) = (s_1-s_2)^2 \int_0^1 g''\bigl(
 	(1-t)s_1 + t s_2
 	\bigr)
 	\bigl(
 	t\log t + (1-t) \log (1-t)
 	\bigr) dt.
 \end{equation*} 
 Thus, assuming that $g''$ is uniformly bounded, we obtain the estimate
 \begin{equation*}
 	|\CB(a; g)|\le  C \|g''\|_{\plainL\infty} 
 	\iint\frac{|a(\xi_1)-a(\xi_2)|^2}{|\xi_1-\xi_2|^2}d\xi_1 d\xi_2,
 \end{equation*}
 with a universal constant $C>0$, which guarantees the finiteness of $\CB(a; g)$ 
 under the condition \eqref{sobolev:eq}. 
However, in applications one often needs 
 non-smooth functions, see e.g. \cite{GiKl}, 
 \cite{HLS}, \cite{LeSpSo}, \cite{LeSpSo1} and references therein. 
 The main aim of this paper is to investigate the coefficient \eqref{cb:eq} 
 for real-valued symbols $a$ and non-smooth 
functions $g:\R\mapsto \mathbb C$, 
described in Condition \ref{f:cond} further on. A representative example 
of one such 
function is $g(t) = |t|^\g$ with some $\gamma\in (0, 1]$.  
Surprisingly, even finiteness of $\CB(a; g)$ for such a function is far from trivial. 
The main result 
(see Theorem \ref{bbound:thm}) is a bound 
on the coefficient $\CB(a; g)$ 
that explicitly depends on the symbol $a$ and function $g$.  
Formula \eqref{Widom_82:eq} for non-smooth functions $g$ is proved in \cite{LSS3}.

Henceforth by $C$ and $c$ with or without indices we denote various positive constants 
whose precise value is of no importance. 
The value of constants may vary from line to line.

 \section{Smooth functions $g$}

 Before embarking on the formulation of the main theorem we provide some 
 useful information on the smooth case.   
 First we show how to extend formula \eqref{cb_pv:eq} to $\plainC{1, \varkappa}$-functions.  
 Rewrite $U(s_1, s_2; g)$ in a different way introducing the integral
 \begin{equation}\label{V:eq}
 V(s_1, s_2; g) = \int_0^1 \frac{g\bigl((1-t)s_1 + t s_2\bigr)
 	- g(s_2)}{1-t} dt.
 \end{equation}
 This functional is well-defined for any $\varkappa$-H\"older continuous function 
 $g: \mathbb C\to\mathbb C$ with $\varkappa\in (0, 1]$, and 
 \begin{equation}\label{V_bd:eq}
 |V(s_1, s_2; g)|\le C_\varkappa \1 g\1_{\plainC{0,\varkappa}} |s_1-s_2|^\varkappa,\ 
 \quad\forall s_1, s_2\in\mathbb C,
 \end{equation}
 where we have denoted
 \begin{equation*}
 	\1 g\1_{\plainC{0,\varkappa}} = \sup_{z, w\in \mathbb C, z\not = w} \frac{|g(z)-g(w)|}{|z-w|^\varkappa}. 
 \end{equation*}
 If $g$ is boundedly 
 differentiable, then, integrating by parts once, we obtain 
 \begin{equation}\label{V_diff:eq}
 V(s_1, s_2; g) = (s_2-s_1)\int_0^1 \log(1-t)g'\bigl((1-t)s_1 + t s_2\bigr)dt.
 \end{equation}
 Due to the elementary formula
 \begin{align*}
 	&\ \frac{g((1-t)s_1 + ts_2) - (1-t) g(s_1) - tg(s_2)} {t(1-t)} \\[0.2cm]
 	&\ \qquad\qquad\qquad = \frac{g((1-t)s_1 + ts_2) - g(s_1)}{t} 
 	+ \frac{g((1-t)s_1 + ts_2) - g(s_2)}{1-t},
 \end{align*}
 we have 
 \begin{equation}\label{VU:eq}
 	U(s_1, s_2; g) = V(s_2, s_1; g) + V(s_1, s_2; g), 
 \end{equation}
 so that in combination with \eqref{V_diff:eq} we obtain
 \begin{equation}\label{VtoU:eq}
 U(s_1, s_2; g)= (s_2-s_1)\int_0^1 \log(1-t)\bigl[
 g'\bigl((1-t)s_2 + t s_1\bigr)
 - g'\bigl((1-t)s_1 + t s_2\bigr)
 \bigr]dt.
 \end{equation}
 
 \begin{lem}
 	Suppose that $g'$ is $\varkappa$-H\"older 
 	continuous with some  $\varkappa\in (0, 1]$. Then  
 	\begin{equation}\label{gagliardo:eq}
 		|\CB(a; g)|\le C 
 		\1 g'\1_{\plainC{0,\varkappa}}
 		\iint \frac{|a(\xi_1) - a(\xi_2)|^{1+\varkappa}}{|\xi_1-\xi_2|^2} d\xi_1 d\xi_2,
 	\end{equation}
 	with a universal constant $C$.
 \end{lem} 
 
 \begin{proof}
 	Since 
 	\begin{align*}
 		|g'\bigl((1-t)s_2 + t s_1\bigr)
 		- g'\bigl((1-t)s_1 + t s_2\bigr)|
 		\le &\ |1-2t|^\varkappa \1 g'\1_{\plainC{0,\varkappa}}
 		|s_1-s_2|^\varkappa,\\[0.2cm]
 		\le &\ \1 g'\1_{\plainC{0,\varkappa}}
 		|s_1-s_2|^\varkappa,\ \forall t\in [0, 1], 
 	\end{align*}
 	formula \eqref{VtoU:eq} gives 
 	\begin{equation*}
 		|U(s_1, s_2; g)|\le C 
 		\1 g'\1_{\plainC{0,\varkappa}}|s_1-s_2|^{1+\varkappa},\ 
 		C = -\int_0^1 \log(1-t) dt.
 	\end{equation*}
 	This leads to the proclaimed bound. 
 \end{proof}
 
 The double integral in \eqref{gagliardo:eq} is the standard Gagliardo-Slobodetski seminorm of 
 $a$ in $\plainW{s}{p}(\R)$ raised to power $p$, where $p = 1+\varkappa$, 
 and $s = (1+\varkappa)^{-1}$, see e.g. \cite{NPV}.
 
 For the next theorem we rewrite the definition \eqref{cb:eq} of the coefficient 
 $\CB(a; g)$ as the principal value integral: 
 \begin{equation}\label{cb_pv:eq}
 \CB(a; g) = \frac{1}{8\pi^2}\lim_{\varepsilon\downarrow 0}	
 \CB_\varepsilon(a; g),\ \quad 
 \CB_\varepsilon(a; g)
 = 
 \underset{|\xi_1-\xi_2|>\varepsilon}\iint
 \frac{U\bigl(a(\xi_1), a(\xi_2); g\bigr)}{|\xi_1-\xi_2|^2}
 d\xi_1 d\xi_2.
 \end{equation}
 In view of \eqref{VU:eq},
  \begin{equation}\label{cb2_pv:eq}
 \CB_\varepsilon(a; g) = \frac{1}{4\pi^2}
 \underset{|\xi_1-\xi_2|>\varepsilon}\iint
 \frac{V\bigl(a(\xi_1), a(\xi_2); g\bigr)}{|\xi_1-\xi_2|^2}
 d\xi_1 d\xi_2.
 \end{equation}
 This representation can be transformed into a different formula for 
 the coefficient $\CB(a; g)$, known in the literature, see e.g. \cite{Widom_85}, Proposition 5.4 
 or \cite{BuBu}, formula (1.5). 
 For any $m\in\R$ and $n= 0, 1, 2\dots,$ denote
 \begin{equation*}
 	\|u\|_{m}^{(n)} = \max_{0\le k\le n} \sup_{\xi\in\R} (1+|\xi|)^{m+k} |u^{(k)}(\xi)|.
 \end{equation*}
 
 \begin{thm}\label{redV:thm}
 	Suppose that $g', g''\in\plainL{\infty}(\R)$, and that $\|a\|_{m}^{(2)}<\infty$ 
 	with some $m\in (0, 1)$. Then   
 	the limit  \eqref{cb_pv:eq} exists and it is given by 
 	\begin{equation}\label{new:eq}
 	\CB(a; g) = \frac{1}{4\pi^2}\int\lim_{\varepsilon\to 0}\underset{|\xi_1-\xi_2|>\varepsilon}\int 
 	\frac{g(a(\xi_1)) - g(a(\xi_2))}
 	{a(\xi_1) - a(\xi_2)} \frac{a'(\xi_1)}{\xi_1-\xi_2} d\xi_1 d\xi_2. 
 	\end{equation}
 	Moreover, 
 	\begin{equation*}
 		|\CB(a; g)|\le C 
 		\bigl[\|g'\|_{\plainL\infty} \|a'\|_{m+1}^{(1)} 
 		+ \|g''\|_{\plainL\infty} (\|a'\|_{m+1}^{(0)})^2
 		\bigr],
 	\end{equation*}
 	with a constant $C>0$ independent of the functions $a$ and $g$. 
 \end{thm}
 
 Before proving the above formula we point out some useful properties of the 
 Hilbert transform
 \begin{equation}\label{Hilbert:eq}
 \tilde u(\xi) = \frac{1}{\pi} \lim_{\varepsilon\to 0}
 \underset{|\eta-\xi|>\varepsilon}\int \frac{u(\eta)}{\eta-\xi} d\eta, 
 \end{equation}
 derived in \cite{Widom_85}, Lemmas 5.2, 5.3.

 \begin{prop} \label{Hilbert:prop}
 	Suppose that $\|u\|_m^{(1)}<\infty$ 
 	for some $m \in (0, 1)$. 
 	Then  
 	\begin{equation*}
 		|\tilde u(\xi)|\le C \|u\|_m^{(1)}(1+|\xi|)^{-m}.
 	\end{equation*}
 	If, in addition, 
 	$\|u\|_{m+1}^{(1)}<\infty$ and 
 	\begin{equation*}
 		\int u(\eta) d\eta = 0,
 	\end{equation*}
 	then 	
 	\begin{equation*}
 		|\tilde u(\xi)|\le C \|u\|_{m+1}^{(1)}(1+|\xi|)^{-m-1}. 
 	\end{equation*}
 	The constants in the above inequalities do not depend on $u$. 
 \end{prop}

 \begin{proof}[Proof of Theorem \ref{redV:thm}] 
 	First we check that the integral on the right-hand side is finite. 
 	Observe that 
 	\begin{align}\label{derV:eq}
 		u(\xi_1; \xi_2):=\frac{\p}{\p \xi_1}V\bigl(a(\xi_1), a(\xi_2); g\bigr)
 		= &\ a'(\xi_1) \int_0^1 g'\bigl((1-t) a(\xi_1) + ta(\xi_2)\bigr)dt\notag\\[0.2cm]
 		= &\ a'(\xi_1)\frac{g\bigl(a(\xi_1)\bigr) - g\bigl(a(\xi_2)\bigr)}{a(\xi_1) - a(\xi_2)},
 	\end{align}
 	so that 
 	\begin{equation*}
 		\| u(\cdot; \xi_2)\|_{m+1}^{(1)} 
 		\le C\bigl(\|g'\|_{\plainL\infty} \|a'\|_{m+1}^{(1)} 
 		+ \|g''\|_{\plainL\infty} (\|a'\|_{m+1}^{(0)})^2
 		\bigr),
 	\end{equation*}
 	uniformly in $\xi_2\in\R$.
 	Moreover, 
 	\begin{equation*}
 		\int u(\xi_1; \xi_2)d\xi_1 = 0,  
 	\end{equation*}
 	and consequently, by Proposition \ref{Hilbert:prop}, 
 	\begin{equation*}
 		|\tilde u(\eta; \xi_2)|\le C(1+|\eta|)^{-m-1} 
 		\bigl(\|g'\|_{\plainL\infty} \|a'\|_{m+1}^{(1)} 
 		+ \|g''\|_{\plainL\infty} (\|a'\|_{m+1}^{(0)})^2
 		\bigr),
 	\end{equation*}
 	uniformly in $\xi_2\in\R$, where $\tilde u(\eta; \xi_2)$ denotes the Hilbert transform 
 	of the function $u(\xi_1, \xi_2)$ in the variable $\xi_1$. 
 	Since 
 	\begin{equation*}
 		\CB(a; g) = \frac{1}{4\pi} \int \tilde u(\xi_2, \xi_2) d\xi_2,
 	\end{equation*}
 	this leads to the required estimate.
 	
 	Now we concentrate on the derivation of 
 	\eqref{new:eq}. To this end integrate \eqref{cb2_pv:eq} by parts:
 	\begin{align}
 		4\pi^2\CB_\varepsilon(a; g)
 		= &\ \frac{1}{\varepsilon}\int \bigl[
 		V\bigl(a(\xi_2 + \varepsilon), a(\xi_2); g\bigr) + 
 		V\bigl(a(\xi_2 - \varepsilon), a(\xi_2); g\bigr)
 		\bigr]d\xi_2 \notag\\[0.2cm]
 		+ &\ \int \underset{|\xi_1-\xi_2|>\varepsilon}\int
 		\frac{1}{\xi_1-\xi_2} 
 		\frac{\p }{\p\xi_1}V\bigl(a(\xi_1), a(\xi_2); g\bigr)
 		d\xi_1 d\xi_2.\label{beps:eq}
 	\end{align}  
 	By \eqref{derV:eq}, the double integral on the right-hand side 
 	of \eqref{beps:eq} coincides with the one in \eqref{new:eq}. 
 	To handle the first integral on the right-hand side of 
 	\eqref{beps:eq}, note that by \eqref{V_bd:eq} with 
 	$\varkappa = 1$,
 	\begin{equation*}
 		\frac{|V\bigl(a(\xi_2 \pm \varepsilon), a(\xi_2); g\bigr)|}{\varepsilon}
 		\le C \|g'\|_{\plainL\infty}\frac{|a(\xi_2\pm \varepsilon) - a(\xi_2)|}{\varepsilon}
 		\le C\|g'\|_{\plainL\infty}\|a\|_{m}^{(1)}(1+|\xi_2|)^{-m-1},
 	\end{equation*} 
 	uniformly in $\varepsilon\in (0, 1]$, and that 	 
 	\begin{equation*}
 		\lim_{\varepsilon\to 0} \frac{V\bigl(a(\xi_2 \pm \varepsilon), a(\xi_2); g\bigr)}{\varepsilon}
 		= \pm a'(\xi_2) g'\bigl(a(\xi_2)\bigr) 
 		= \pm \frac{d}{d \xi_2} g\bigl(a(\xi_2)\bigr).
 	\end{equation*}
 	Clearly, the integral of the right-hand side equals zero. 
 	Thus by the Dominated Convergence Theorem the 
 	first term on the right-hand side of 
 	\eqref{beps:eq} tends to zero as $\varepsilon\to 0$, and the formula  \eqref{new:eq} 
 	is proved. 	 
 \end{proof}

\section{Non-smooth functions}

\subsection{Main result}
We concentrate on the very special non-smooth case, which is nonetheless interesting for applications. 
To distinguish from smooth functions, we change the notation from $g$ to $f$ and assume that $f$ 
satisfies the following condition: 

\begin{cond}\label{f:cond} 
For some integer $n \ge 1$, some $\g \in (0, 1]$ and some $x_0\in \R$, the function 
$f\in\plainC{n}(\R\setminus\{ x_0 \})\cap\plainC{}(\R)$ satisfies the 
bound 
\begin{equation}\label{fnorm:eq}
\1 f\1_n = 
\max_{0\le k\le n}\sup_{x\not = x_0} |f^{(k)}(x)| |x-x_0|^{-\g+k}<\infty.
\end{equation}
\end{cond}

The constants in all subsequent estimates may 
depend on $n, \gamma$, but not on $x_0$. 

For a function $f$ satisfying the above condition 
the bound holds: 
\begin{equation}\label{fbound:eq}
|f^{(k)}(x)| \le \1 f\1_n |x-x_0|^{\g-k},  
k = 0, 1, \dots,  n, \ \quad  \   x\not = x_0, 
\end{equation}
If $n\ge 1$, then the above condition implies that $f$ is 
$\g$-H\"older continuous, and in particular,  
\begin{equation}\label{hol:eq}
|f(x_1) - f(x_2)|\le 2\1 f\1_1 |x_1-x_2|^\g, \ \forall x_1, x_2\in\R.
\end{equation}
For a function $u$ denote 
\begin{equation*}
	\SN\bigl(u; \plainl{\d}(\plainW{N}{p})\bigr) 
	= \biggl[\sum_{n\in\Z} \max_{0\le k\le N} 
	\biggl( \underset{(n, n+1)}\int |u^{(k)}(\xi)|^p d\xi\biggr)^{\frac{\d}{p}}\biggr]^{\frac{1}{\d}}, 
\end{equation*}
where $\d\in (0, \infty]$, $p\in (0, \infty]$. 
Now we can state the main result.

\begin{thm}\label{bbound:thm}
	Suppose that the function $f:\R\to\mathbb C$ 
	satisfies Condition \ref{f:cond} with $n = 2$, and some $\g\in (0, 1]$, 
	$x_0\in \R$. Let $a$ be a real-valued function such that 
$a\in \plainW{N}{p}_{\textup{\tiny loc}}(\R)$  with some $p\in (1, \infty]$ 
	and some $N$ such that $N\ge \g^{-1}+p^{-1}$.  
	Then the limit \eqref{cb_pv:eq} exists and it 
	satisfies the bound 
	\begin{equation}\label{bbound:eq}
		|\CB(a; f)|\le C_\g \1 f\1_1\underset{|\xi_1-\xi_2|>1}
		\iint \frac{|a(\xi_1) - a(\xi_2)|^\g}{|\xi_1-\xi_2|^2} d\xi_1 d\xi_2
		+ C_\g \1 f\1_2 \ \biggl[\SN\bigl(a'; \plainl{\g}(\plainW{N-1}{p})\bigr)\biggr]^\g,
	\end{equation}
	where the constant $C_\g$ is independent of the functions $f$, 
	$a$, and the parameter $x_0$.  
\end{thm}

Note that the value of the right-hand side of 
\eqref{bbound:eq} is preserved under the shift 
$a\to a+a_0$ with an arbitrary constant $a_0$.  
If we assume that $a-a_0\in \plainL{\g}(\R)$ with 
some constant $a_0$, then the first integral 
in \eqref{bbound:eq} can be estimated as follows:  
\begin{equation*}
	\underset{|\xi_1-\xi_2|>1}
	\iint \frac{|a(\xi_1) - a(\xi_2)|^\g}{|\xi_1-\xi_2|^2} d\xi_1 d\xi_2
	\le 2\underset{|\xi_1-\xi_2|>1}
	\iint \frac{|a(\xi_1)-a_0|^\g}{|\xi_1-\xi_2|^2} d\xi_1d\xi_2\le 
	4 \int |a(\xi)-a_0|^\g d\xi.
\end{equation*}

\subsection{Function $f$} 
Here we prove some elementary properties of the function 
$f$ satisfying Condition \ref{f:cond} with $n=2$.

\begin{lem}\label{fdash:lem}
If $\g \in (0, 1]$, then for any $t_1\not = x_0, t_2\not = x_0$, and any $\d\in [0, 1]$, we have 
\begin{equation}\label{fdash:eq}
|f'(t_1) - f'(t_2)|\le 2^{1-\d} \1 f\1_2(\min_{j=1, 2}|t_j-x_0|)^{\g-1-\d}|t_1-t_2|^\d.
\end{equation}
\end{lem}

\begin{proof} 
	Suppose that either  
$t_1 > x_0, t_2 < x_0$, or $t_1 < x_0, t_2 > x_0$.  
According to \eqref{fbound:eq}, for any $\d>0$ we have 
\begin{align}
	|f'(t_1)| \le &\ \1 f\1_1 |t_1-x_0|^{\g-1}\label{fdashup:eq}\\[0.2cm]
	= &\ \1 f\1_1 |t_1-x_0|^{\g-1-\d} |t_1-x_0|^{\d}
	\le \1 f\1_1 (\min_{j=1, 2}|t_j-x_0|)^{\gamma-1-\d}|t_1-t_2|^\d,\notag
\end{align}
Estimating $f'(t_2)$ in the same way we get the claimed bound. 

Suppose now that $t_2\ge t_1>x_0$ or $t_2\le t_1<x_0$. 
Then 
\begin{equation*}
	|f'(t_1) - f'(t_2)|\le |f''(\t)| |t_1-t_2|,\ \ \ \textup{with some}
	\ \ \t\in(t_1, t_2),
\end{equation*}
and hence, by \eqref{fbound:eq},  
\begin{equation*}
|f'(t_1) - f'(t_2)|\le \1 f\1_2 |t_1-x_0|^{\gamma-2} |t_1-t_2|.
\end{equation*}
Together with \eqref{fdashup:eq}, this gives 
\begin{equation*}
|f'(t_1) - f'(t_2)|\le \1 f\1_2 2^{1-\d}|t_1-x_0|^{(\gamma-1)(1-\d)}
|t_1-x_0|^{(\gamma-2)\d} |t_1-t_2|^\d,   
\end{equation*}
for any $\d\in [0, 1]$. 
This leads to \eqref{fdash:eq}, as claimed. 

The cases $t_1>t_2>x_0$ or $t_1<t_2<x_0$ are handled by exchanging the roles of $t_1$ and $t_2$.
\end{proof}

\subsection{Functional $V$}
Let us derive some useful estimates for the functional $V$ defined in \eqref{V:eq}. 
As before, we assume that $f:\R\mapsto\mathbb C$ in the definition
\eqref{cb_pv:eq} satisfies Condition \ref{f:cond} 
with some $\g\in (0, 1]$, $n= 2$ and $x_0\in\R$.

First we make some straightforward observations. 
In view of \eqref{hol:eq} and \eqref{V_bd:eq},  
\begin{equation}\label{V_est:eq}	|V(s_1, s_2; f)|\le 
C_\g \1 f\1_1|s_1-s_2|^\g.
\end{equation}
Furthermore, by definition \eqref{V:eq} and by \eqref{hol:eq}, 
for any $\mu\in (0, 1)$, we have 
\begin{align}\label{approx:eq}
	|V(s_1, s_2; f) - V(r_1, r_2; f)|\le &\ C\1 f\1_1 
	|\log\mu|\bigl(|s_1-r_1|^\g + |s_2-r_2|^\g\bigr) \notag\\[0.2cm]
	&\ + C\1 f\1_1 \mu^\g \bigl(|s_1-s_2|^\g + |r_1-r_2|^\g\bigr)|,
\end{align}
for any real $s_1, r_1, s_2, r_2$. This bound follows 
from \eqref{V:eq} by splitting $V$ into two integrals: 
over $(0, 1-\mu)$ and over $(1-\mu, 1)$.

Now introduce  
\begin{align}
	Y(s_1, s_2; f) = &\ \p_{s_1}V(s_1, s_2; f)
	= 	\int_0^1 f'\bigl(s_1(1-t) + s_2t\bigr)dt,\label{Ydef:eq}\\[0.2cm]
	X(s_1, s_2; f) = &\ Y(s_1, s_2; f) - f'(s_1),\ s_1\not = x_0.\label{Xdef:eq} 
\end{align}

\begin{lem}
	Let $f$ satisfy Condition \ref{f:cond} with 
	$\gamma\in (0, 1]$, $n=2$ and $x_0\in\R$, 
	and let $\d\in [0, \g)$ be some number. 
	Then for all real $s_1\not = x_0$ and all real $s_2$,
	\begin{equation}\label{Xab:eq}
	|X(s_1, s_2; f)|\le 2^{2-\d} (\g-\d)^{-1}
	\1 f\1_2 |s_1-s_2|^{\d} |s_1-x_0|^{\g-1-\d}.
	\end{equation}
	%
\end{lem}

\begin{proof} 
	Represent $X$ in the form
	\begin{equation*}
		X(s_1, s_2; f)= 
		\int_0^1 \bigl[f'\bigl((1-t)s_1 + t s_2\bigr)
		- f'(s_1)\bigr] dt.
	\end{equation*} 
	First suppose that either 
	$s_1 > x_0, s_2 \ge x_0$, or $s_1 < x_0, s_2 \le x_0$.
	Then, by \eqref{fdash:eq}, 
	\begin{align*}
		\bigl|f'\bigl((1-t)s_1 + t s_2\bigr)
		- &\  f'(s_1 )\bigr|\\[0.2cm]
		\le &\ 2^{1-\d}\1 f\1_2 (1-t)^{\g-1-\d} t^\d
		|s_1-x_0|^{\g - 1 -\d}|s_1-s_2|^\d, 
	\end{align*}
	for any $\d\in [0, 1]$. Consequently,
	\begin{equation*}
		|X(s_1, s_2; f)|\le
		2^{1-\d}\1 f\1_2|s_1- s_2|^{\d}  
		|s_1-x_0|^{\g-1-\d}
		\int_0^1 (1-t)^{\g-1-\d} t^{\d}dt.
	\end{equation*}
	The integral is finite for $\d\in [0, \g)$, which leads to \eqref{Xab:eq}.
	
	Now suppose that either  
	$s_1 > x_0, s_2 < x_0$, or $s_1 < x_0, s_2 > x_0$.  
	According to \eqref{fdash:eq}, 
	\begin{align*}
		\bigl|f'\bigl((1-t)s_1  +&\  t s_2\bigr) - f'(s_1)
		\bigr|\\[0.2cm]
		\le &\ 2^{1-\d}\1 f\1_2 
		|(1-t)s_1+ts_2-x_0|^{\g - 1 -\d} t^\d|s_1-s_2|^\d\\[0.2cm]
		= &\ 2^{1-\d}\1 f\1_2 |s_1-s_2|^{\g-1-\d}
		\biggl|
		t-\frac{s_1-x_0}{s_1-s_2}
		\biggr|^{\g-1-\d}t^\d|s_1-s_2|^\d, \ \ t\not = \frac{s_1-x_0}{s_1-s_2}.
	\end{align*}
	Since $\g\le 1$ and $|s_1-s_2|> |s_1-x_0|$, we estimate
	\begin{equation*}
		|s_1-s_2|^{\g-1-\d} < |s_1-x_0|^{\g-1-\d}.
	\end{equation*} 
	Furthermore,
	\begin{equation*}
		\int_0^1 |t-z|^{\g-1-\d} dt \le \frac{2}{\g-\d}
	\end{equation*}
	uniformly in $z\in [0, 1]$. This implies \eqref{Xab:eq}.
\end{proof}

\section{Two lemmas on integrals of polynomials}
 In this section we prepare two elementary results involving real-valued 
 polynomial functions $a$. 

For a closed interval $I\subset \R$ we denote 
by $|I|$ its length (the Lebesgue measure). For a 
smooth function $a$ on $I$ we denote 
by $\|a\|_{\plainL{p}}$ its $\plainL{p}$-norm on the interval $I$.

\begin{lem} \label{zeros:lem} 
	Let $I\in\R$ be a closed interval, and let $a$ be a real-valued 
	polynomial. 
	Suppose that $I$ contains at least 
	$N-1$ distinct critical points of the function $a$, 
	with some $N = 1, 2, \dots$. Let $p\in [1, \infty]$ be arbitrary.  Then for any $\g\in (0, 1]$ 
	and any two points 
	$\eta_1, \eta_2\in I$ the bound holds 
	\begin{equation}\label{zeros:eq}
	||a(\eta_1)|^\g - |a(\eta_2)|^\g|\le \| a^{(N)}\|_{\plainL{p}}^{\g} |I|^{\g (N-\frac{1}{p})}.
	\end{equation} 
	If $I$ contains exactly $N-1$ distinct critical points of $a$, then 
	the total variation $\textup{Var}[|a|^\g; I]$ 
	of the function $|a|^\g$ on the interval $I$  
	satisfies the bound
	\begin{equation}\label{agamma:eq}
	\textup{Var}[|a|^\g; I]\le (N+1)^2 \| a^{(N)}\|_{\plainL{p}}^{\g} |I|^{\g (N-\frac{1}{p})}.
	\end{equation} 
	\end{lem}

\begin{proof} 
	Assume without loss of generality that $\|a^{(N)}\|_{\plainL{p}}\le 1$. 
	Since the interval $I$ contains at least $N-1$ distinct zeros of $a'$,  
	by an elementary argument, 
	the interval $I$ also contains at least $N-2$ distinct zeros of $a''$, 
	$N-3$ distinct zeros of $a'''$, and eventually, 
	at least one point $\xi_0$, such that $a^{(N-1)}(\xi_0) = 0$. 
	This means that
	\begin{equation*}
	|a^{(N-1)}(\xi)|\le \int_{\xi_0}^{\xi}
	|a^{(N)}(\eta)| d\eta\le 
	\|a^{(N)}\|_{\plainL{p}} |I|^{1-\frac{1}{p}}\le |I|^{1-\frac{1}{p}},\ \forall \xi\in I.
	\end{equation*} 
	From this bound we obtain consecutively that 
	$ |a^{(N-2)}(\xi)|\le |I|^{2-\frac{1}{p}}$, $|a^{(N-3)}(\xi)|\le |I|^{3-\frac{1}{p}}$, and in general, 
	$|a^{(k)}(\xi)|\le |I|^{N-k-\frac{1}{p}}$, $k = 1, 2, \dots, N-1$. 
	In particular, $|a'(\xi)|\le |I|^{N-1-\frac{1}{p}}$, so that for any $\eta_1, \eta_2\in I$ we have 
	\begin{equation*}
	a(\eta_1) - a(\eta_2) = w(\eta_1, \eta_2), \ 
	|w(\eta_1, \eta_2)|\le |I|^{N-1-\frac{1}{p}} |\eta_1-\eta_2|
	\le |I|^{N-\frac{1}{p}}. 
	\end{equation*}
	Thus 
	\begin{equation*}
	||a(\eta_1)|^\g - |a(\eta_2)|^\g|\le |w(\eta_1, \eta_2)|^\g\le |I|^{\g(N-\frac{1}{p})},
	\end{equation*} 
	as claimed.

	In order to prove \eqref{agamma:eq},  note that the polynomial $a$ 
	has at most $N$ distinct roots on $I$, and hence 
	there are at most $N+1$ intervals where the polynomial $a$ is sign-definite. Using the additivity 
	of total variation, it suffices to prove that on each of these 
	intervals the total variation does not exceed 
	$  (N+1)\| a^{(N)}\|_{\plainL{p}}^{\g} |I|^{\g (N-\frac{1}{p})}$. 
	Assume for simplicity that $a(\xi)\ge 0$ for all $\xi\in I$. 
	Partition $I$ into 
	intervals $\{I_{j}\}$ on which the function $a$ is monotone. Thus by \eqref{zeros:eq},
	\begin{equation*}
	\textup{Var}[|a|^\g; I_{j}]\le \| a^{(N)}\|_{\plainL{p}}^{\g} |I|^{\g (N-\frac{1}{p})}.
	\end{equation*}
	As the number of intervals $I_j$ does not exceed 
	$N$, we immediately obtain the required bound. 
\end{proof}

\begin{lem} \label{zeros2:lem} 
	Let $I\in\R$ be a closed interval, such that $|I|\le r$ with some number 
	$r>0$, and let 
	$a$ be a real-valued polynomial. 
	Let $\gamma\in (0, 1]$, $p\in [1, \infty]$ and $N \ge \gamma^{-1}+p^{-1}$.  
	Then the total variation $\textup{Var}[|a|^\g; I]$ 
	of the function $|a|^\g$ on the interval $I$  
	satisfies the bound
	\begin{equation}\label{agamma2:eq}
	\textup{Var}[|a|^\g; I]\le C_\g (N+1)^2\| a'\|_{\plainW{N-1}{p}}^{\g} 
	|I|^{(1-p^{-1})\g},  
	\end{equation} 
	and hence, 
	\begin{equation}\label{agamma3:eq}	
	\int_I |a'(\xi)| |a(\xi)|^{\g-1} d\xi \le C_\g (N+1)^2\| a'\|_{\plainW{N-1}{p}}^{\g} 
	|I|^{(1-p^{-1})\g},
	\end{equation}
	with a constant $C_\g = C_\g(r)$ independent of $a$ and $N$. 
\end{lem}

\begin{proof} 
	Let $I_k, k = 1, 2, K,$ be non-empty closed intervals with disjoint interiors 
	such that $I = \cup_k I_k$, and satisfying the following 
	requirements: 
	\begin{itemize}
		\item
		each $I_k$, $k = 1, 2, \dots, K-1,$ contains exactly  $N-1$ critical points of 
		$a$, 
		\item
		the interval $I_K$ contains no more than $N-1$ critical points of $a$. 
	\end{itemize}
	By \eqref{agamma:eq}, for any $k = 1, 2, \dots, K-1$ we have 
	\begin{align}\label{ik:eq}
	\textup{Var}[|a|^\g; I_k]
	\le &\ (N+1)^2\| a^{(N)}\|_{\plainL{p}}^{\g} |I_k|^{\g (N-\frac{1}{p})}\notag\\[0.2cm]
	\le &\ r^{\g (N-\frac{1}{p})-1}(N+1)^2\| a^{(N)}\|_{\plainL1}^{\g} |I_k|,  
	\end{align}
	where we have used that $\g (N-p^{-1}) \ge 1$. 
	Furthermore, by \eqref{agamma:eq} again, 
	\begin{align}\label{klast:eq}
	\textup{Var}[|a|^\g; I_K]\le &\ (L+1)^2\| a^{(L)}\|_{\plainL{p}}^{\g} 
	|I|^{(L-\frac{1}{p})\g},\notag\\[0.2cm]
	\le &\ r^{(L-1)\g}(N+1)^2\| a'\|_{\plainW{N-1}{p}}^{\g} 
	|I|^{(1-\frac{1}{p})\g}, 
	\end{align}
	where $L-1\le N-1$ is the number of critical points on $I_K$. 
	By the additivity, the inequalities \eqref{ik:eq} and \eqref{klast:eq} 
	lead to \eqref{agamma2:eq}. 
	The left-hand side of \eqref{agamma2:eq} coincides with that 
	of \eqref{agamma3:eq}. This completes the proof. 	
\end{proof}

 \section{Proof of Theorem \ref{bbound:thm}}

We begin the proof of Theorem \ref{bbound:thm} with estimating $\CB_1(a; f)$, which 
will produce the integral term on the right-hand side 
of \eqref{bbound:eq}.  The function $f$ is assumed to satisfy Condition \ref{f:cond}. 
As before, all constants in the estimates below are independent 
of the symbol $a$, function $f$, parameter $x_0$, 
but may depend on $\g\in (0, 1]$ and other relevant parameters unless otherwise stated. 

\begin{lem}\label{B1:lem} 
Assume that $f$ is as specified above. Then 
 	\begin{equation*}
 	|\CB_1(a; f)|\le  C_\g \1 f\1_1 \underset{|\xi_1-\xi_2|>1}
 	\iint \frac{|a(\xi_1) - a(\xi_2)|^\g}{|\xi_1-\xi_2|^2} d\xi_1 d\xi_2.
 	\end{equation*}
\end{lem}

 \begin{proof} The required bound  immediately follows from \eqref{cb2_pv:eq} 
 	and \eqref{V_est:eq}.
 \end{proof}
 
The remaining part of the coefficient $\CB(a; f)$ is studied with the help of 
a suitable partition of unity on $\R$.  
For a function $\zeta\in\plainC\infty_0(\R)$ and numbers $R>0$, 
$\varepsilon\in (0, R)$, define  
\begin{equation}\label{CD:eq}
\CD_{\varepsilon, R}(a; \zeta, f) = \frac{1}{4\pi^2}
\underset{\varepsilon<|\xi_1-\xi_2|< R}\iint
\zeta(\xi_1)\frac{V\bigl(a(\xi_1), a(\xi_2); f\bigr)}{|\xi_1-\xi_2|^2}
d\xi_1 d\xi_2.
\end{equation}
In all the subsequent bounds the constants are independent of the 
cut-off $\z$, and of the parameters $\varepsilon, R$.

\begin{thm} \label{CD:thm}
	Let $\zeta\in\plainC1_0(-1, 1)$, and let $a\in\plainW{N}{p}(-2, 2)$ 
	with some $p\in (1, \infty]$ and $N\ge \g^{-1}+p^{-1}$. Then 
	\begin{equation}\label{CD1:eq}
		|\CD_{\varepsilon, R}(a; \zeta, f)|
		\le C_{\g, \d}  
		\| \z\|_{\plainC1} \1 f\1_2 R^{(1-\frac{1}{p})\d} A_{N, p}(a)^\g,
	\end{equation}		
	for any $\d\in (0, \g)$, uniformly 
	in $R\in (0, 1]$ and $\varepsilon\in (0, R]$.
Here
\begin{equation}\label{dernorm:eq}
A_{N, p}(a) = \|a'\|_{\plainW{N-1}{p}},
\end{equation}
where the norm is taken on the interval $(- 2, 2)$. 

Furthermore, the limit of $\CD_{\varepsilon, R}(a; \zeta, f)$ as $\varepsilon\to 0$, 
exists. 
\end{thm}

Note the following straightforward estimate:
 \begin{equation}\label{Lp:eq}
 |a(\xi_1) - a(\xi_2)|\le 
 A_{1, p}(a)
 |\xi_1-\xi_2|^{1-\frac{1}{p}}, \ \xi_1, \xi_2\in (-2, 2). 
 \end{equation}

Integrating \eqref{CD:eq} by parts we get:
  \begin{equation}\label{splitCD:eq}
  \CD_{\varepsilon, R}(a; \zeta, f) 
  =  \CD_{\varepsilon, R}^{(1)}(a; \zeta, f)
  + \CD_{\varepsilon, R}^{(2)}(a; \zeta, f)
  + \CD_{\varepsilon}^{(3)}(a; \zeta, f)
  - \CD_{R}^{(3)}(a; \zeta, f)
  \end{equation}
  with
  \begin{align*}
  \CD_{\varepsilon, R}^{(1)} (a; \zeta, f)
  = &\ \frac{1}{4\pi^2}
  \underset{\varepsilon<|\xi_1-\xi_2|< R}\iint
  \frac{\zeta(\xi_1)}{\xi_1-\xi_2} 
  \frac{\p }{\p\xi_1}V\bigl(a(\xi_1), a(\xi_2); f\bigr)
  d\xi_1 d\xi_2,\\[0.2cm]
  \CD_{\varepsilon, R}^{(2)} (a; \zeta, f)
  = &\ \frac{1}{4\pi^2}
  \underset{\varepsilon<|\xi_1-\xi_2|< R}\iint
  \frac{V\bigl(a(\xi_1), a(\xi_2); f\bigr)}{\xi_1-\xi_2} 
  \frac{\p }{\p\xi_1}\zeta(\xi_1)
  d\xi_1 d\xi_2, 
  \end{align*}
  and 
  \begin{align} 
  	\CD_{\varepsilon}^{(3)} (a; \zeta, f)
  	= &\ \frac{1}{4\pi^2 \varepsilon}
  	\int \bigl[
  	\zeta(\xi+\varepsilon)
  	V\bigl(a(\xi + \varepsilon), a(\xi); f\bigr)\notag\\[0.2cm]
  	&\ \quad\quad\quad + 
  	\zeta(\xi-\varepsilon)V\bigl(a(\xi - \varepsilon), a(\xi); f\bigr)
  	\bigr]d\xi, \label{D3:eq}
  \end{align}
 Below we estimate each term separately. 
   
  \begin{lem}\label{D2:lem}
  	Suppose that $\zeta\in\plainC\infty_0(-1, 1)$ and that 
  	$a\in\plainW{1}{p}(-2, 2)$, 
  	$p\in (1, \infty]$. 
  	Then 
  	\begin{equation}\label{D2:eq}
  	|\CD_{\varepsilon, R}^{(2)}(a; \zeta, f)|\le 
  	C_\gamma \1 f\1_1\max|\zeta'| R^{(1-\frac{1}{p})\g} A_{1, p}(a)^\g, 
  	\end{equation}
  	uniformly in $ R\in (0, 1]$ and $\varepsilon\in(0, R]$.
  \end{lem}

 \begin{proof}
 By \eqref{V_est:eq} and \eqref{Lp:eq} we have: 
  	\begin{equation*}
  		|V\bigl(a(\xi_1), a(\xi_2); f\bigr)|
  		\le C\1 f\1_1 
  		A_{1, p}(a)^\g |\xi_1-\xi_2|^{(1-\frac{1}{p})\g}, 
  	\end{equation*}
  	so that \eqref{D2:eq} follows immediately. 
  \end{proof}
  
 For the next group of results we need to assume that $a$ 
 is a real-valued polynomial.   
    
  \begin{lem} 
Suppose that $\zeta\in\plainC\infty_0(-1, 1)$, and that $a$ is a real-valued 
polynomial. Then 
\begin{equation*}
	|\CD_{\varepsilon}^{(3)}(a; \zeta, f)|\le 
	C_{\g, \d} \|\z\|_{\plainC1} \1 f\1_2 \varepsilon^{(1-\frac{1}{p})\d}
	 A_{N, p}(a)^\gamma, 
		\end{equation*}
		for any $\d\in [0, \g)$, $p\in [1, \infty]$ and any $N\ge \g^{-1}+p^{-1}$,  
	uniformly in $\varepsilon\in(0, 1]$. The constant $C_{\g, \d}$ may depend on the 
	parameter $N$. 
\end{lem}
  
  \begin{proof} Without loss of generality assume that $\|\z\|_{\plainC1}=1$.  
  	Represent:
  	\begin{align*}
  	\int	\zeta(\xi\pm\varepsilon)
  		V\bigl(a(\xi\pm \varepsilon), &\ a(\xi); f\bigr) d\xi\\
  = &\ \pm\int \int_0^\varepsilon \bigl[\zeta'(\xi\pm\nu) 
  V\bigl(a(\xi\pm \nu), a(\xi); f\bigr) \\[0.2cm]
  &\ \quad \quad\quad + \zeta(\xi\pm\nu) a'(\xi\pm \nu) Y
  \bigl(a(\xi\pm \nu), a(\xi); f\bigr)
  \bigr]d\nu d\xi\\[0.2cm]
= &\ \pm\int \int_0^\varepsilon \bigl[\zeta'(\xi) 
V\bigl(a(\xi), a(\xi\mp \nu); f\bigr) \\[0.2cm]
&\ \quad\quad\quad + \zeta(\xi) a'(\xi) Y
\bigl(a(\xi), a(\xi\mp\nu); f\bigr)
\bigr]d\nu d\xi,
  	\end{align*}
  	see \eqref{Ydef:eq} for the definition of the function $Y$. 
Let us simplify the formula for 
$\CD_\varepsilon^{(3)}$, introducing the integrals 
  	\begin{align*}
  	S_1^{(\pm)}
  	= &\ \frac{1}{\varepsilon}
  	\int  \zeta'(\xi) \int_0^\varepsilon
  	V\bigl(a(\xi), a(\xi\mp \nu); f\bigr) d\nu d\xi,\\[0.2cm]
  	S_2^{(\pm)}
  	= &\ \frac{1}{\varepsilon}\int  \zeta(\xi) \int_0^\varepsilon 
   a'(\xi) X\bigl(a(\xi), a(\xi\mp\nu); f\bigr)d\nu d\xi,
  	\end{align*}
  see \eqref{Xdef:eq} for the definition of $X$. Therefore
  \begin{equation*}
  4\pi^2 \CD_\varepsilon^{(3)}(a; \z, f)= S_1^{(+)} - S_1^{(-)} + S_2^{(+)}
  - S_2^{(-)}.
  \end{equation*}
  By \eqref{V_est:eq} and \eqref{Lp:eq},
\begin{align*}
|S_1^{(\pm)}|
\le &\ 
\frac{C}{\varepsilon} \1 f\1_1 \int |\z'(\xi)| \int_0^\varepsilon
|a(\xi) - a(\xi-\nu)|^\g d\nu d\xi\\[0.2cm]
\le &\ 
\frac{C}{\varepsilon} \1 f\1_1  A_{1, p}(a)^\g
\int_0^\varepsilon \nu^{(1-\frac{1}{p})\g} d\nu
\le 
C \1 f\1_1   A_{1, p}(a)^\g \varepsilon^{(1-\frac{1}{p})\g}.
\end{align*}  	
  	To estimate $S_2^{(\pm)}$ use \eqref{Xab:eq} with $\d\in [0, \gamma)$ 
  	and \eqref{Lp:eq} again: 
  	\begin{align*}
|X\bigl(a(\xi), a(\xi\mp\nu); f\bigr)|\le &\ 
C \1 f\1_2 |a(\xi)- a(\xi\mp\nu)|^{\d} |a(\xi)-x_0|^{\g-1-\d}\\[0.2cm]
\le  &\ C \1 f\1_2 A_{1, p}(a)^\d |\nu|^{(1-\frac{1}{p})\d} |a(\xi)-x_0|^{\g-1-\d}. 	
  	\end{align*}
  	Therefore 
\begin{align*}
	|S_2^{(\pm)}|
	\le \frac{C }{\varepsilon} A_{1, p}(a)^\d\1 f\1_2 
	\int_0^\varepsilon \nu^{(1-\frac{1}{p})\d} d\nu
	\int_{-1}^{1} |a'(\xi)| |a(\xi)-x_0|^{\g-1-\d}d\xi.
\end{align*}
By virtue of \eqref{agamma3:eq}, 
\begin{equation*}
|S_2^{(\pm)}|\le CA_{1, p}(a)^\d \1 f\1_2\ 
\varepsilon^{(1-\frac{1}{p})\d}
 A_{N, p}(a)^{\gamma-\d}.
\end{equation*}
Since $A_{1, p}(a)\le A_{N, p}(a)$, the required bound follows. 
\end{proof}

  \begin{lem}\label{D1:lem}
  		Suppose that $\zeta\in\plainC\infty_0(-1, 1)$ and that $a$ is a real-valued 
  		polynomial. 
  		Then 
  \begin{equation}\label{D1:eq}
  |\CD_{\varepsilon, R}^{(1)}(a; \z, f)|
  \le C_{\g, \d}\1 f\1_2 \max|\z|  R^{(1-\frac{1}{p})\d} A_{N, p}(a)^{\g},
  \end{equation}
  for any 
   $\d\in [0, \g)$, $p\in (1, \infty]$ and any $N\ge \g^{-1}+p^{-1}$, 
   uniformly in $R\in (0, 1]$ and $\varepsilon\in (0, R]$.   
  \end{lem}
  
  \begin{proof} Without loss of generality assume that $\max|\z|=1$.  
  	Since 
  	\begin{equation*}
  	 \underset{\varepsilon<|\xi_1-\xi_2|< R}\iint
  		\frac{1}{\xi_1-\xi_2} \z(\xi_1)a'(\xi_1) g'\bigl(a(\xi_1)\bigr)
  		d\xi_1 d\xi_2 = 0,
  	\end{equation*}
  	the integral $\CD_{\varepsilon, R}^{(1)}$ can be rewritten as 
  	\begin{equation*}
  		\underset{\varepsilon<|\xi_1-\xi_2|< R}\iint
  		\frac{1}{\xi_1-\xi_2} \z(\xi_1)a'(\xi_1) X\bigl(a(\xi_1), a(\xi_2); f\bigr) d\xi_1 d\xi_2,  
  	\end{equation*}
  	see \eqref{Xdef:eq} for the definition of the function $X$.
  	
  	By virtue of \eqref{Xab:eq} and \eqref{Lp:eq}, 
  	for any $\d\in [0, \g)$ the integrand is bounded from above by 
  	\begin{align*}
  	C_{\g, \d}\1 f\1_2 &\ \frac{|a(\xi_1)-a(\xi_2)|^\d}{|\xi_1-\xi_2|}
  	|a'(\xi_1)|
  	|a(\xi_1) - x_0|^{\g-1-\d}\\[0.2cm]
  	\le &\ 
  	C_{\g, \d}\1 f\1_2  A_{1, p}(a)^{\d}|\xi_1-\xi_2|^{(1-p^{-1})\d-1}
  	|a'(\xi_1)|
  	|a(\xi_1) - x_0|^{\g-1-\d},
  	\end{align*}
  	for all $\xi_1$ where $a(\xi_1)\not = x_0$.
  	Assuming that $\d >0$, and using \eqref{agamma3:eq}, 
  	we obtain that 
  	\begin{equation*}
  	|\CD_{\varepsilon, R}^{(1)}(a; \z, f)|
  	\le C_{\g, \d}\1 f\1_2  R^{(1-\frac{1}{p})\d}
  	A_{1, p}(a)^\d A_{N, p}(a)^{\g-\d}.
  	\end{equation*}
  	As $A_{1, p}\le A_{N, p}$, the bound \eqref{D1:eq} follows. 
  \end{proof}

 \begin{proof}[Proof of Theorem \ref{CD:thm}]

 Collecting the bounds established in 
Lemmas \ref{D2:lem}-\ref{D1:lem}, and using the representation 
\eqref{splitCD:eq}, we arrive at the bound 
 \eqref{CD1:eq} for a polynomial $a$. 

For an arbitrary function $a\in\plainW{N}{p}(-2,2)$, $p\in (1, \infty]$, and a number 
$q \le p$, $1<q<\infty$, find a polynomial $\tilde a = \tilde a_\varepsilon$,  
such that 
\begin{equation}\label{polyapprox:eq}
	\|a-\tilde a\|_{\plainW{N}{q}} < A_{1, p}(a)R^{\g^{-1}}\varepsilon^{4\g^{-1}}.
\end{equation}  
This implies that 
\begin{equation}\label{ANq:eq}
A_{N, q}(\tilde a)\le A_{N, q}(a) + A_{1, p}(a)\le A_{N, p}(a)
\bigl(
4^{\frac{1}{q}-\frac{1}{p}} + 1
\bigr).	
\end{equation}
For subsequent calculations we assume without loss of generality that $\1 f\1_2 = 1$ and 
$\|\z\|_{\plainC1} = 1$. 
In view of \eqref{approx:eq}, for any $\mu\in (0, 1)$ we have
\begin{align*}
\bigl|V(a(\xi_1), a(\xi_2); f) - &\ V(\tilde a(\xi_1), \tilde a(\xi_2); f)\bigr|\\[0.2cm]
\le &\ C |\log\mu|\bigl( 
|a(\xi_1)-\tilde a(\xi_1)|^\g + |a(\xi_2) - \tilde a(\xi_2)|^\g\bigr)\\[0.2cm]
&\ + C \mu^\g \bigl( A_{1, q}(\tilde a)^\g |\xi_1-\xi_2|^{(1-\frac{1}{q})\g}
+ A_{1, p}(a)^\g |\xi_1-\xi_2|^{(1-\frac{1}{p})\g}
\bigr),
\end{align*}
where we have also used \eqref{Lp:eq}. 
Consequently,
\begin{align*}
	|\CD_{\varepsilon, R}(a; \zeta, f)
	- &\ \CD_{\varepsilon, R}(\tilde a; \zeta, f)|\\[0.2cm]
	\le &\ \frac{C}{\varepsilon^2} 
	\bigl[|\log\mu| \|a-\tilde a\|_{\plainL{q}}^\g 
	+ \mu^\g A_{1, p}(a)^\g R
	\bigr]\\[0.2cm]
	\le &\ \frac{C}{\varepsilon^2}  R A_{1, p}(a)^\g
	\bigl(|\log\mu|  \varepsilon^4 + \mu^\g \bigr),
	\end{align*}
	where we have used \eqref{polyapprox:eq}. 
Take $\mu = \varepsilon^{3\g^{-1}}$, so that 
 	 \begin{equation}\label{CDdiff:eq}
 	 	|\CD_{\varepsilon, R}(a; \zeta, f)
 	 	-  \CD_{\varepsilon, R}(\tilde a; \zeta, f)|
 	 	\le C  R  A_{1, p}(a)^\g \varepsilon. 
 	 \end{equation}
 	 Let $\tilde\d$ be given by 
 	  	 \begin{equation*}
 	  	 	\tilde\d = \d\  \frac{1-p^{-1}}{1-q^{-1}},
 	  	 \end{equation*}
 	  	 where $\d\in (0, \g)$. 
 	  	 By picking a suitable $q$ one ensures that $\tilde\d < \g$ 
 	  	 as well. 
 	  	 Now use Theorem \ref{CD:thm} for the polynomial $\tilde a$ with the parameter $\tilde \d$ 
 	  	 instead of $\d$, remembering \eqref{ANq:eq}:
 	  	 \begin{equation*}
 	  	 |\CD_{\varepsilon, R}(\tilde a; \z, f)|\le 
 	  	 C  R^{(1-\frac{1}{q})\tilde\d}A_{N, q}(\tilde a)^\g
 	  	 \le  C  R^{(1-\frac{1}{p})\d}A_{N, p}(a)^\g.
 	  	 \end{equation*}
 	  	 Combining this bound with \eqref{CDdiff:eq} we obtain \eqref{CD1:eq}.
 	  
 	 Finally, the existence of the limit
 	 \begin{equation*}
 	 	\underset{\varepsilon\to 0}\lim \CD_{\varepsilon, R}(a; \z, f) 
 	 \end{equation*}
follows from the fact that the right-hand side of \eqref{CD1:eq} tends to zero as 
$R\to 0, \varepsilon\to 0$.  	 
 	 \end{proof} 
    
\begin{proof}[Proof of Theorem \ref{bbound:thm}] 
	Let $\zeta_k\in\plainC\infty_0(\R)$, $k\in\Z$, be a family of functions 
	constituting a partition of unity subordinate to the covering of the 
	real axis by intervals $(k-1, k+1), k\in\Z$. We may assume that 
	the norms $\|\zeta_k\|_{\plainC1}$ are bounded uniformly in $k\in\Z$. 
	Represent $\CB_\varepsilon(a; f)$ as 
	\begin{equation*}
	\CB_{\varepsilon}(a; f)= \CB_1(a; f) + \sum_{k\in\Z} \CD_{\varepsilon, 1}(a; \z_k, f).
	\end{equation*}
	The first term on the right-hand side is estimated by Lemma \ref{B1:lem}. 
	Due to the bound \eqref{CD1:eq} the second term is 
	bounded by $C \1 f\1_2 \SN(a'; \plainl{\g}(\plainW{N-1}{p}))^\g$.
	Furthermore, since the $\SN$-(quasi)-norm is finite, the sum has a limit as $\varepsilon\to 0$.
This completes the proof. 
\end{proof}
    
\section{A special case}

In the previous Section, in the proof of Theorem \ref{bbound:thm}, we use 
the covering of the real axis by intervals $(k-1, k+1), k\in\Z$ that obviously all have 
length $2$. Now we derive an estimate for $\CB(a; f)$ using 
a covering by intervals whose size is sensitive to the 
rate of change of the function $a$.  
Let us describe in more precise terms the conditions on $a$.  
Let $\tau:\R\to\R$ be a positive function satisfying the condition 
\begin{equation}\label{Lip:eq}
|\tau(\xi) - \tau(\eta)| \le \nu |\xi-\eta|,\ \ \forall\xi, \eta\in\R, 
\end{equation}
with some $\nu \in (0, 1)$. 
It is straightforward to check that 
\begin{equation}\label{dve:eq}
(1+\nu)^{-1}\le \frac{\tau(\xi)}{\tau(\eta)} \le (1-\nu)^{-1},\ \ 
\forall \eta\in J(\xi) = \bigl(\xi-\tau(\xi), \xi+\tau(\xi)\bigr).
\end{equation}
We call $\tau$ \textit{the scale function}. 
Let $v:\R\to\R$ be another continuous positive function such that 
\begin{equation}\label{w:eq}
C_1 \le \frac{v(\eta)}{v(\xi)}\le C_2,\ \forall \eta\in J(\xi),
\end{equation}
with some positive constants $C_1, C_2$ independent 
of $\xi$ and $\eta$.  
We call $v$ \textit{the amplitude function}. 
Since $\nu < 1$, one can construct a covering of $\R$ by open intervals $J(\xi_j)$  
centred at some points $\xi_j, j\in\Z$, which satisfies 
\textit{the finite intersection property,} i.e.  the number of intersecting intervals 
is bounded from above  
by a constant depending only on the parameter $\nu$, see \cite{Hor}, Chapter 1. 
Moreover, there exists a partition of unity $\phi_j\in\plainC\infty_0(\R)$ 
subordinate to the 
above covering such that 
\begin{equation}\label{partition:eq}
|\phi_j^{(k)}(\xi)|\le C_k \tau(\xi)^{-k},\ k = 0, 1, \dots, 
\end{equation}
with some constants $C_k$ independent of $j \in\Z$.

It is convenient for us to use a covering with finite intersection property, 
constructed with the help of the function $\tau/2$ instead of $\tau$ itself. 
Let
\begin{equation*}
	I_j = \biggl(\eta_j - \frac{\tau_j}{2},
	\eta_j+ \frac{\tau_j}{2}
	\biggr),\ \ \tau_j = \tau(\eta_j), j\in\Z,
\end{equation*}
be intervals forming such a covering, and let $\phi_j\in\plainC\infty_0(\R), j\in\Z$, be 
a subordinate partition of unity satisfying \eqref{partition:eq}.

Consider a symbol $a\in \plainC{N}(\R)$, satisfying the 
bounds
\begin{equation}\label{scales:eq}
|a(\xi)-a_0|\le Cv(\xi),\ \ 
| a^{(k)}(\xi)|\le C_k \tau(\xi)^{-k} v(\xi),\ k = 1, 2, \dots, N,
\end{equation}
with some functions $\tau$ and  $v$ described above, and with some constant 
$a_0$.

In all the bounds below the constants are independent of the functions 
$f$, $\tau$ and $v$, but 
may depend on the parameter $\nu$ and the 
constants in \eqref{w:eq} and \eqref{scales:eq}. 

\begin{thm}\label{coeffscales:thm} 
	Suppose that $f$ satisfies Condition \ref{f:cond} with $n = 2$ and $\gamma\in (0, 1]$. 
	Let $\tau, v, a$ satisfy \eqref{Lip:eq}, \eqref{w:eq} and let 
	$a$ satisfy 
\eqref{scales:eq} with some $N\ge \g^{-1}$. Then 
\begin{equation}\label{coeffscales:eq}
|\CB(a; f)|\le C_\g\1 f\1_2 \int \frac{v(\xi)^\g}{\tau(\xi)} d\xi.
\end{equation}
\end{thm}

A similar bound holds also for functions $f$ with higher smoothness. 

\begin{thm}\label{coeffscales1:thm}
Suppose that $f:\mathbb C\to\mathbb C$ is a function such that $f'$ 
is $\varkappa$-H\"older continuous with some $\varkappa\in (0, 1]$. 
Let $\tau, v, a$ satisfy \eqref{Lip:eq}, \eqref{w:eq} and let 
	$a$ satisfy 
\eqref{scales:eq} with $N = 1$. Then 
\begin{equation}\label{coeffscales1:eq}
|\CB(a; f)|\le C_\varkappa
\| f'\|_{\plainC{0, \varkappa}} \int \frac{v(\xi)^{1+\varkappa}}{\tau(\xi)} d\xi.
\end{equation}
\end{thm}

First we give a detailed proof of Theorem \ref{coeffscales:thm}. 

Represent $\CB(a; f)$ as follows:
\begin{equation}\label{scales_rep:eq}
\CB_\varepsilon(a; f) = \sum_{j\in\mathbb Z} 
\CD_{\varepsilon, \infty}(a; \phi_j, f),
\end{equation}
see \eqref{CD:eq} for the definition of $\CD_{\varepsilon, R}(\cdots)$.  
Split each summand into two components: 
\begin{equation*}
\CD_{\varepsilon, \infty}(a; \phi_j, f)
= 
\CD_{\varepsilon, R_j}(a; \phi_j, f) 
+ \CD_{R_j, \infty}(a; \phi_j, f),\ \ \ R_j = \frac{\tau_j}{2}.
\end{equation*} 
 
\begin{lem} \label{outside:lem}
Suppose that the scaling function $\tau$ 
satisfies \eqref{Lip:eq} with some $\nu\in (0, 1)$. 
If $f$ satisfies the conditions of Theorem \ref{coeffscales:thm}, then 
\begin{equation}\label{outside:eq}
\sum_{j\in\Z}|\CD_{R_j, \infty}(a; \phi_j, f)|\le 
C\1 f\1_1 \int \frac{|a(\xi)-a_0|^\g}{\tau(\xi)} d\xi.
\end{equation}
\end{lem}

\begin{proof} 
Assume without loss of generality that $\1 f\1_1=1$ and $a_0= 0$. 
For all $\xi_1\in I_j$ 
we get from \eqref{Lip:eq} that
\begin{equation*}
\tau(\xi_1)\le \frac{2+\nu}{2}\tau_j.
\end{equation*}
Thus for all $\xi_2$ such that $|\xi_1-\xi_2|>\tau_j/2$ we have
\begin{align*}
	\tau(\xi_1)\le &\ (2+\nu)|\xi_1-\xi_2|,\\[0.2cm]
\tau(\xi_2)\le &\ \nu|\xi_1-\xi_2| + \tau(\xi_1)
\le 2(\nu+1)|\xi_1-\xi_2|,
\end{align*}
which leads to 
\begin{equation*}
c_\nu\bigl(\tau(\xi_1) + \tau(\xi_2)\bigr)\le |\xi_1-\xi_2|,
c_\nu = \frac{1}{4(1+\nu)}. 
\end{equation*}
Therefore
\begin{equation*}
|\CD_{R_j, \infty}(a; \phi_j, f)|\le 
\underset{|\xi_1-\xi_2|> c_\nu(\tau(\xi_1)+\tau(\xi_2))}\iint \phi_j(\xi_1)
\frac{|V(a(\xi_1), a(\xi_2); f)|}{|\xi_1-\xi_2|^2} d\xi_1 d\xi_2.
\end{equation*}
By \eqref{V_est:eq}, 
the right-hand side does not exceed
\begin{equation*}
C\underset{|\xi_1-\xi_2|> c_\nu\tau(\xi_1)}\iint \phi_j(\xi_1)
\frac{|a(\xi_1)|^\g}{|\xi_1-\xi_2|^2} d\xi_1 d\xi_2
+ C\underset{|\xi_1-\xi_2|> c_\nu\tau(\xi_2)}\iint \phi_j(\xi_1)
\frac{|a(\xi_2)|^\g}{|\xi_1-\xi_2|^2} d\xi_1 d\xi_2. 
\end{equation*}
Thus the sum over $j$ is bounded from above by
\begin{equation*}
C\underset{|\xi_1-\xi_2|> c_\nu\tau(\xi_1)}\iint 
\frac{|a(\xi_1)|^\g}{|\xi_1-\xi_2|^2} d\xi_1 d\xi_2
+ C\underset{|\xi_1-\xi_2|> c_\nu\tau(\xi_2)}\iint 
\frac{|a(\xi_2)|^\g}{|\xi_1-\xi_2|^2} d\xi_1 d\xi_2. 
\le C'\int \frac{|a(\xi)|^\g}{\tau(\xi)} d\xi,
\end{equation*}
as claimed. 
\end{proof}
 
\begin{lem}\label{inside:lem}
Let $a$ satisfy \eqref{scales:eq} 
with some functions $\tau = \tau(\xi)$ and $v = v(\xi)$ 
satisfying \eqref{Lip:eq} and \eqref{w:eq}. 
Suppose also that $N\ge \g^{-1}$ and $R\le R_j$. 
Then for any $\d\in [0, \g)$ the bound holds
\begin{equation}\label{inside:eq}
|\CD_{\varepsilon, R}(a; \phi_j, f)|	\le 
	C_\d\1 f\1_2 R^\d \int_{I_j}  
	\frac{v(\xi)^\g}{\tau(\xi)^{1+\d}} d\xi,
\end{equation}
uniformly in $\varepsilon\in (0, R]$.
\end{lem}  
  
\begin{proof} 
	Without loss of generality assume $\1 f\1_2 = 1$. 
Let 
\begin{equation*}
	\tilde a(\eta) = a(\eta_j+ R_j\eta),\ 
	\tilde{\phi}_j(\eta) = \z_j(\eta_j+R_j\eta).
\end{equation*}
Thus by \eqref{partition:eq}, $\|\tilde{\phi}_j\|_{\plainC1}\le C$,\ 
$\supp\tilde \phi_j\subset(-1, 1)$ uniformly in $j$, and in view of \eqref{w:eq}, 
\eqref{scales:eq},
\[
|\tilde a^{(n)}(\eta)|\le C_n v(\eta_j),\ \forall \eta: |\eta|\le 2,
\]
for all $n = 1, \dots, N$, so that 
$A_{N, \infty}(\tilde a)\le C v(\eta_j)$, see \eqref{dernorm:eq} for the 
definition.  
 Thus by Theorem \ref{CD:thm} with $p=\infty$, and arbitrary $\d\in [0, \g)$,
\begin{equation*}
|\CD_{\varepsilon, R_j}(a; \phi_j, f)| = 
|\CD_{\varepsilon R_j^{-1}, R R_j^{-1}}(\tilde a, \tilde \phi_j, f)|\le C (R R_j^{-1})^\d v(\eta_j)^\g.
\end{equation*}
The right-hand side is trivially estimated by 
\begin{equation*}
C R^\d \int_{I_j} v(\eta_j)\tau_j^{-1-\d} d\xi.	
\end{equation*}
By virtue of \eqref{Lip:eq} and \eqref{w:eq}, this 
is bounded by the right-hand side of \eqref{inside:eq}. 
This completes the proof. 
\end{proof}  
 
\begin{cor}\label{sub:cor}
Suppose that $\tau_{\textup{\tiny inf}} = \inf \tau(\xi)>0$, and that 
$R\le \tau_{\textup{\tiny inf}}/2$. Then 
for any $\d\in [0, \g)$ the bound holds:
\begin{equation}\label{sub:eq}
|\CD_{\varepsilon, R}(a; 1, f)|\le 
	C_\d\1 f\1_2 R^\d \int \frac{v(\xi)^\g}{\tau(\xi)^{1+\d}} d\xi,
\end{equation}
uniformly in $\varepsilon \in (0, R]$. 
\end{cor}

\begin{proof}[Proof of Corollary 
\ref{sub:cor} and Theorem \ref{coeffscales:thm}] 
Since the covering $\{I_j\}$ possesses the finite intersection property,  
the bound \eqref{sub:eq} follows from the bound \eqref{inside:eq} by summing over all $j$'s.

Using the bound \eqref{inside:eq} with $R = R_j$ we obtain that
\begin{equation*}
|\CD_{\varepsilon, R_j}(a; \phi_j, f)|	\le 
	C\1 f\1_2 \int_{I_j}  
	\frac{v(\xi)^\g}{\tau(\xi)} d\xi,
\end{equation*}
uniformly in $\varepsilon\in (0, 1]$. 
In view of the finite intersection property of the covering $\{I_j\}$, we get
\begin{equation*}
\sum_{j\in\Z}|\CD_{\varepsilon, R_j}(a; \phi_j, f)|
\le C\1 f\1_2 \int \frac{v(\xi)^\g}{\tau(\xi)} d\xi.
\end{equation*}	
In view of the representation \eqref{scales_rep:eq} this bound 
together with \eqref{outside:eq} lead to \eqref{coeffscales:eq}. 
\end{proof}

For Theorem \ref{coeffscales1:thm} we give only a sketch of the proof.
The details are either the same as in the preceding proof, or they 
can be easily filled in. 

\begin{proof}[Sketch of the proof of Theorem \ref{coeffscales1:thm}] 
By \eqref{gagliardo:eq}, the proof reduces to estimating 
the integral
\begin{equation*}
\iint \frac{|a(\xi_1) - a(\xi_2)|^{1+\varkappa}}{|\xi_1-\xi_2|^2} d\xi_1 d\xi_2.
\end{equation*}
As in Lemma \ref{outside:lem} one can show that 
\begin{equation*}
\sum_j  \underset{|\xi_1-\xi_2|> R_j}
\iint \phi_j(\xi_1)\frac{|a(\xi_1) - a(\xi_2)|^{1+\varkappa}}{|\xi_1-\xi_2|^2} d\xi_1 d\xi_2
\le C\int \frac{|a(\xi)-a_0|^{1+\varkappa}}{\tau(\xi)}d\xi
\end{equation*}
Furthermore, if $|\xi_1-\xi_2|<R_j$, $\xi_1\in I_j$, then by \eqref{scales:eq}, 
\eqref{Lip:eq} and \eqref{w:eq}, 
\begin{equation*}
|a(\xi_1) - a(\xi_2)|\le Cv(\eta_j) \tau_j^{-1}|\xi_1-\xi_2|, 
\end{equation*}
so that
\begin{align*}
\underset{|\xi_1-\xi_2|< R_j}
\iint \phi_j(\xi_1)&\ \frac{|a(\xi_1) - a(\xi_2)|^{1+\varkappa}}{|\xi_1-\xi_2|^2} d\xi_1 d\xi_2\\[0.2cm]
\le &\ Cv(\eta_j)^{1+\varkappa} \tau_j^{-1-\varkappa}\int_{I_j}   \int_{|\xi_1-\xi_2|<R_j} |\xi_1-\xi_2|^{\varkappa-1} d\xi_2 d\xi_1\\[0.2cm]
\le &\ Cv(\eta_j)^{1+\varkappa} \le C\int_{I_j} v(\eta_j)^{1+\varkappa}\tau_j^{-1}d\xi. 
\end{align*}
Now, as in the proof of Lemma \ref{inside:lem}, the last integral is bounded by 
$\int v^{1+\varkappa}\tau^{-1} d\xi$. This 
completes the proof of Theorem \ref{coeffscales1:thm}.
\end{proof}

We illustrate the usefulness of the bound \eqref{coeffscales:eq} with the example of 
the symbol 
\begin{equation}\label{positiveT:eq} 
a(\bxi) = a_{T}(\bxi) = \frac{1}{1+ \exp{\frac{\xi^2 - \mu}{T}}},
\end{equation}
where $T\in (0, T_0]$, $T_0 >0$ and $\mu\in\R$ 
are some parameters. This symbol is nothing but 
the Fermi function for non-interacting Fermions at positive temperature 
$T$ and chemical potential $\mu$, see e.g. \cite{LeSpSo1}. 
We are interested in the small $T$ behaviour, whereas 
the value $\mu$ is kept fixed. Assume for simplicity that $\mu = 1$.  
It is clear that in a neighbourhood of 
the points $\xi = \pm 1$ the derivatives 
of $a$ grow as $T\to 0$.  
It is straightforward to check that 
\begin{equation}\label{an:eq}
|a^{(n)}(\xi)|\le C_n a(\xi) (1-a(\xi))(1+|\xi|)^n T^{-n}, n = 1, 2, \dots,
\end{equation}
and 
\begin{equation}\label{a1-a:eq}
a(\xi) \bigl(1-a(\xi))\le \exp{\biggl(-\frac{|\xi^2 - 1|}{T}\biggr)},\ \xi\in\R. 
\end{equation}
Thus Theorem \ref{bbound:thm} with any $p \in (1, \infty]$ leads to the estimate
\begin{equation}\label{prim:eq}
| \CB(a; f)|\le C\1 f\1_1 + \tilde C\1 f\1_2 T^{-N\g + \frac{\g}{p}}.
\end{equation}
The right-hand side is greater than $C T^{-1}$, since $N\ge \g^{-1} + p^{-1}$.

Let us now estimate $\CB(a_T; f)$ in a different way, by applying 
Theorem \ref{coeffscales:thm}. 
Since
\begin{equation*}
(1+|\xi|)^n T^{-n} \exp{\biggl(-\frac{|\xi^2 - 1|}{2T}\biggr)}
\le C_n (||\xi|-1| + T )^{-n},\ 
C_n = C_n(T_0),
\end{equation*}
in view of \eqref{an:eq} and \eqref{a1-a:eq}, we have 
\begin{equation*} 
|a^{(n)}(\xi)|\le C_n  (||\xi|-1| + T )^{-n} \exp{\biggl(-\frac{|\xi^2 - 1|}{2T}\biggr)}, 
n = 1, 2, \dots.
\end{equation*}
This shows that $a$ satisfies \eqref{scales:eq} with
\begin{equation*}
a_0 = 0,\ \  
\tau(\xi) = \frac{1}{2}(||\xi|-1| + T ),\ \ v(\xi) = v_\b(\xi) = (1+|\xi|)^{-\b},
\end{equation*}
with an arbitrary $\b >0$. Consequently, by Theorem \ref{coeffscales:thm}, 
\begin{align*}
|\CB(a; f)|\le &\ C\1 f\1_2 \int (||\xi|-1| + T )^{-1} (1+|\xi|)^{-\b\g}d\xi\\[0.2cm]
\le &\ C\1 f\1_2 \bigl(|\log T| + 1\bigr).
\end{align*}
 This bound is clearly sharper than \eqref{prim:eq}, and its precision (as $T\to0$) 
 is confirmed by the 
 asymptotic formula for $\CB(a_T; f)$, $T\to 0$, announced in \cite{LeSpSo1}. 
 
\textit{Acknowledgement} 
The author is grateful to W. Spitzer for the careful 
reading of the manuscript, and for useful remarks.

\bibliographystyle{amsplain}

\begin{thebibliography}{99}

\bibitem{Basor} E. Basor,
\emph{Trace formulas for Toeplitz matrices with piecewise continuous symbols}, 
J. Math. Anal. Appl \textbf{120}(1986),25--38. 

 
 
\bibitem{BuBu} 
A. M. Budylin, V.S. Buslaev, \emph{On the asymptotic 
behaviour of the spectral characteristics 
of an integral operator with a difference kernel on 
expanding domains}, Differential equations, Spectral theory, 
Wave propagation (Russian), 16-–60, 305, Probl. Mat. Fiz., 13, Leningrad. Univ., Leningrad, 1991. 
 
\bibitem{GiKl} D. Gioev, I. Klich,
\emph{Entanglement Entropy of fermions in any dimension and the Widom Conjecture},
Phys. Rev. Lett. \textbf{96} (2006), no. 10, 100503, 4pp.


 \bibitem{Hor} 
 L. H\"ormander, \emph{The Analysis of Linear Partial Differential Operators, Vol. 1}, 
 Springer, New York 1993.
 


\bibitem{HLS}R.C. Helling, H. Leschke, W. Spitzer,
\emph{
A special case of a conjecture by Widom with implications
to fermionic entanglement entropy}, 
Int. Math. Res. Notices vol. 2011 (2011), pp 1451-1482.





   
\bibitem{LeSpSo} 
H. Leschke, A. V. Sobolev, W. Spitzer, 
\emph{Scaling of R\'enyi entanglement entropies of the free Fermi-gas ground
state: A rigorous proof}, 
Phys. Rev. Lett. \textbf{112}, 160403.
   
   
\bibitem{LeSpSo1} 
   H. Leschke, A. V. Sobolev, W. Spitzer, 
   \emph{Area law for the entanglement entropy of the free Fermi gas at nonzero temperature}, 
   Arxiv: 1501.03412.  
   
\bibitem{LSS3}
H.~Leschke, A.V. Sobolev, W.~Spitzer.
\newblock {Trace formulas for {Wiener--Hopf} operators with applications to
  fermionic entanglement entropy}.
\newblock {\em in preparation}, 2016.
      
\bibitem{NPV} 
E. Di Nezza, G. Palatucci, E. Valdinoci, \emph{Hitchhikerʼs guide 
to the fractional Sobolev spaces}, Bulletin Des Sciences Math\'ematiques, Vol:136 (2012),  521--573.
   
\bibitem{Peller}   
V. V. Peller, \emph{When is a function of a Toeplitz operator close to a Toeplitz operator?}, Toeplitz operators and spectral function theory, 59--85, Oper. Theory Adv. Appl., \textbf{42}, Birkh\"auser, Basel, 1989. 
      
\bibitem{Roc} R. Roccaforte, \emph{Asymptotic expansions of traces for certain convolution operators},
Trans. Amer. Math. Soc., \textbf{285} (1984), 581--602.
 
\bibitem{Sob}
A.V. Sobolev, \emph{Pseudo-differential operators
with discontinuous symbols: Widom's Conjecture}, Memoirs of AMS, 
\textbf{222} (2013), no. 1043.
     


\bibitem{Sob2} A. V. Sobolev, 
\emph{Wiener-Hopf operators in higher dimensions: the Widom conjecture for
	piece-wise smooth domains}, Integr. Equ. Oper. Theory \textbf{81} (2015), Issue 3, 435--449. 


\bibitem{Widom_821} H. Widom,
\emph{On a class of integral
operators with discontinuous symbol},
Toeplitz centennial (Tel Aviv, 1981), pp. 477--500,
Operator Theory: Adv. Appl., 4, Birkh\"auser, Basel-Boston, Mass., 1982.

\bibitem{Widom_80}
H. Widom, 
\emph{Szeg\H o's Limit Theorem: The Higher-Dimensional Matrix Case}, 
Journal of Functional Analysis \textbf{39} (1980), 182--198. 

\bibitem{Widom_82}
H. Widom,
\emph{A trace formula for Wiener-Hopf operators}, 
J. Operator Theory \textbf{8} (1982), 279--298. 

\bibitem{Widom_85} 
H. Widom, \emph{Asymptotic expansions for pseudodifferential operators 
	on bounded domains}, Lecture Notes in Mathematics, V. 1152, Springer, 1985. 

\bibitem{Widom_87}
H. Widom, \emph{
Trace formulas for Wiener-Hopf operators}, 
Operators in indefinite metric spaces, scattering theory and other topics (Bucharest, 1985), 365--371,
Oper. Theory Adv. Appl., \textbf{24}, Birkh\"auser, Basel, 1987. 


 
\end{thebibliography}

\end{document}